\documentclass[11pt,a4paper]{article}
\usepackage{amsmath, amssymb, amsthm} 
\usepackage{graphicx}

\def\<{\langle}
\def\>{\rangle}
\def\RR{\mathbb{R}}
\newcommand\tr{\operatorname{Tr}}
\newcommand\Div{\operatorname{div}}

\newcommand\oc{\overset{\centerdot}}

\def\vol{\,{\rm vol}}

\newtheorem{theorem}{Theorem}
\newtheorem{corollary}{Corollary}
\newtheorem{definition}{Definition}
\newtheorem{example}{Example}

\newtheorem{lemma}{Lemma}
\newtheorem{prop}{Proposition}

\author{Vladimir Rovenski\footnote{Mathematical Department, University of Haifa, Mount Carmel, 3498838 Haifa,  Israel
        \newline e-mail: {\tt vrovenski@univ.haifa.ac.il}        }
}

\title{Godbillon-Vey type functional for almost contact manifolds}

\begin{document}

\date{}

\maketitle

\begin{abstract}
Many contact metric manifolds are critical points of curvature functionals restricted to spaces of associated metrics.
The Godbillon-Vey functional was never considered in a variational context in Contact Geometry.
Recently we extended this functional from foliations to arbitrary plane fields on a 3-dimensio\-nal manifold,
so, the following question arises: {can one use the Godbillon-Vey functional to find optimal almost contact manifolds}?
In~the paper, we introduce a Godbillon-Vey type functional for a 3-dimensional almost contact manifold,
present it in Reinhart-Wood form and find its Euler-Lagrange equations for all variations preserving the Reeb vector field.
We~construct critical (for our functional) 3-dimensional almost contact manifolds having a double-twisted product structure,
these solutions belong to the class $C_{5}\oplus C_{12}$ according to Chinea-Gonzalez classification.


\end{abstract}

\section{Introduction}

D.~Chinea and C.~Gonzalez \cite{CGD-90} decomposed the space of certain 3-tensors on
an almost contact metric manifold $M^{2n+1}(\varphi,\omega,T,g)$ into irreducible invariant components
under the action of the structural group $U(n)\times1$ and developed a Gray-Hervella type classification for almost contact metric (a.c.m.) manifolds.
They obtained 12 classes of a.c.m. manifolds, which in dimension three are reduced to five classes:
$C_5$ of $\beta$-Kenmotsu manifolds, $C_6$ of $\alpha$-Sasakian manifolds, $C_9$-manifolds, $C_{12}$-manifolds
(called generalized Sasakian space forms) and $|C|=C_5\cap C_{12}$ of cosymplectic manifolds.
%
Many works on a.c.m. manifolds are devoted to normal structures, that is to the first three Chinea-Gonzalez classes.
Some authors investigate $C_5\oplus C_{12}$-manifolds, consisting of integrable, non-normal manifolds because $\nabla_TT\ne0$, see~\cite{bbs23,CF-2019,f-2013}, where $\nabla$ is the Levi-Civita connection.
The elements of $\bigoplus_{1\le i\le 5} C_{i}\oplus C_{12}$ are a.c.m. manifolds that are locally a double-twisted product $B \times_{(u,v)} I$,
where $B$ is an almost Hermitian manifold, $I\subset\RR$ is an open interval and $u,v$ are smooth positive functions on $B\times I$.
The~first result in this topic states that any Kenmotsu manifold, i.e., $(\nabla_{X}\,\varphi)Y=g(\varphi X, Y)\,T -\omega(Y)\,\varphi X$,
is locally a warped product $B \times_{u} I$,
where $u\in C^\infty(I)$ and $B$ is a K\"{a}hler manifold, see~\cite{kenmotsu1972}.
Note that any two-dimensional almost complex manifold is complex, moreover, it is K\"{a}hler.
Thus, in dimension three, the class $\bigoplus_{1\le i\le 5} C_{i}\oplus C_{12}$ reduces to
$C_5\oplus C_{12}$, whose elements are locally a double-twisted product $B \times_{(u,v)} I$,
where $B$ is a 2-dimensional K\"{a}hler manifold, see \cite{CF-2019,f-2013}.

\smallskip

Finding critical metrics of certain functionals can be considered as an approach to searching for the best metric for a given manifold.
Many contact metric manifolds are critical points of curvature functionals restricted to spaces of associated metrics;
for example, symplectical manifolds are critical for the integral of scalar curvature and
K-contact manifolds are critical for the integral of Ricci curvature in the $T$-direction, e.g., \cite[Sect.~5]{Blairsurvey}.
The Godbillon-Vey functional was introduced for foliations of codimension one in \cite{gv71},
its changes under infinitesimal deformations of the foliation were studied, for example, in~\cite{as2015},
but this functional was never considered in a variational context in Contact Geometry.
In~\cite{rw3,rw-gv2,R2021} we extended the Godbillon-Vey functional from foliations to arbitrary plane fields on a 3-dimensio\-nal manifold
and used the calculus of variations to characterize critical metrics in distinguished classes of almost-product manifolds.
So, a question arises: \textit{can one use a Godbillon-Vey type functional to find optimal 3-dimensional a.c.m. manifolds,
e.g., in $C_{5}\oplus C_{12}$ according to Chinea-Gonzalez classification}?
To~answer the question, we define a new Godbillon-Vey type functional,
see \eqref{Eq-3gv},
present it in Reinhart-Wood form and derive its Euler-Lagrange equations for variations preserving the Reeb vector field $T$ (Theorem~\ref{T-EL-5}).
For clarity of calculations, we divide variations into two types, each of which has a geometric meaning.
These allow us to find new solutions to the problem: \textit{what a.c.m. manifolds are in some sense optimal}?
Since a.c.m. manifolds with the geodesic vector field $T$ are critical for our functional,
it is interesting to study the case when the curvature $k$ of $T$-curves is non-zero.
We~construct such critical 3-dimensional a.c.m. manifolds having a double-twisted product structure, i.e., solutions belonging to
$C_{5}\oplus C_{12}$ (Theorem~\ref{T-EL-02}).

\section{The Reinhart-Wood type formula}

An almost contact structure $(\varphi,\omega,T)$ on a smooth
manifold $M^{2n+1}$
consists of an endomorphism $\varphi$ of $TM$, a 1-form $\omega\in\Lambda^1(M)$ and a vector field $T$ satisfying
\begin{align}\label{Eq-001}
 \varphi^2(X) = -X +\,\omega(X)\,T\quad (X\in TM),
 \qquad \omega(T)=1.
\end{align}
The plane field $\ker\omega$ is called the contact distribution.
By \eqref{Eq-001}, we get $\omega\circ\varphi=0$ and $\varphi\,T=0$.

Let $T$ be a nonzero vector field on a smooth
orientable three-dimensional manifold $M$.
Then for any 1-form $\omega\in\Lambda^1(M)$ such that $\omega(T)=1$ there exists a unique tensor $\varphi\in End(TM)$
such that the restriction of $\varphi$ on
$\ker\omega$ specifies a right-hand rotation and
$M^3(\varphi,\omega,T)$ is an almost contact manifold.
In \cite{rw3,rw-gv2,R2021}, the Godbillon-Vey functional
 $gv=\int_M \eta\wedge d\eta$, where
 $\eta = \iota_T\,d\omega = d\omega(T,\cdot)$,
was extended from foliations to arbitrary plane fields on $M^3$. Note that~$\eta(T)=0$.
We use the formula $d\omega(X,Y)=X\omega(Y) - Y\omega(X) - \omega([X,Y])$ for $\omega\in\Lambda^1(M)$ and $X,Y\in TM$.

\begin{definition}\rm
Given an almost contact structure $(\varphi,\omega,T)$ on $M^3$, define a 1-form $\eta^*\in\Lambda^1(M)$~by
\begin{align}\label{Eq-00a}
 \eta^*(X) := (\iota_T\,d\omega)(\varphi X) = d\omega(T,\varphi X)\quad (X\in TM),
\end{align}
i.e., $\eta^*=\eta\circ\varphi$.
Using $\eta^*$, we introduce (similarly to $gv$) the following functional:
\begin{align}\label{Eq-3gv}
 gv^*=\int_M \eta^*\wedge d\eta^*.
\end{align}
\end{definition}



If an {almost contact manifold} $M^3(\varphi,\omega,T)$ admits a Riemannian metric $g=\<\cdot,\cdot\>$ such that
\begin{equation}\label{E-Q2}
 \<\varphi X,\varphi Y\> = \<X,Y\> - \omega(X)\,\omega(Y),
\end{equation}
then $g$ is called a compatible metric and we get an a.c.m. manifold.
Such a structure \eqref{E-Q2} is induced on any hypersurface of an almost Hermitian manifold.
Putting $Y=T$ in \eqref{E-Q2}, we get $\omega(X)=\<X,T\>$;
thus, $T$ is $g$-orthogonal to $\ker\omega$.
An a.c.m. manifold such that
\[
 g(X,\varphi Y)=d\omega(X,Y)
\]
is called a contact metric manifold. Such manifolds have a geodesic vector field $T$.

Given a unit vector field $T$ on a Riemannian manifold $(M^3,g)$, the unit normal $N$, the binormal $B = T\times N$ and the {torsion} $\tau$ of $T$-curves
are defined on an open subset $U$ of~$M^3$, where the {curvature} $k$ of $T$-curves is nonzero.
Further, we assume that $U$ is 
dense, so the set $M\setminus U$ can be neglected during integration over $M^3$.
The 1-form in the Godbillon-Vey functional
is given by $\eta = k N^\flat$ (i.e., $\eta(X)=k\<N,X\>$ for $X\in TM$) on $U$ and $\eta=0$ on $M\setminus U$, see~\cite{rw3,rw-gv2}.
Thus, for the 1-form $\eta^*$ in \eqref{Eq-00a}, using $d\omega(T,\varphi X)=-\omega([T,\varphi X])=k\,\<N,\varphi X\>$
and skew-symmetry of $\varphi$, we~get
\begin{align*}
 \eta^* = \left\{\begin{array}{cc}
           -k (\varphi N)^\flat & \mbox{ on } U, \\
            0 &  \mbox{ on } \ M\setminus U. \\
          \end{array}\right.
\end{align*}
If $T$ is a geodesic vector field (i.e., $k=0$) then $\eta=\eta^*=0$, hence $gv^*
=0$.

Recall that the Levi-Civita connection $\nabla$ of a metric $g=\<\cdot,\cdot\>$ is given by:
\begin{equation}\label{E-L-Civita}
 2\,\<\nabla_X\,Y, Z\> = X \<Y,Z\> + Y \<X,Z\> - Z\<X,Y\> +\<[X, Y], Z\> - \<[X, Z], Y\> - \<[Y, Z], X\>.
\end{equation}
The~{non-symmetric second fundamental form} $h$ of the distribution $\ker\omega$ is defined by
\begin{align}\label{E-h-tensor}
 h_{X,Y} = \<\nabla_X Y, T\> = -\<\nabla_X T, Y\>\qquad (X,Y\in \ker\omega) .
\end{align}
If the distribution $\ker\omega$ in $TM^3$ is integrable,
then the tensor
 $2\,{\mathcal T}_{X,Y} = \<[X,Y],T\>$
vanishes; in this case, the 2-form $h$ is symmetric.
The mean curvature of $\ker\omega$ is defined by $H=\frac12\tr_g h$.
The distribution $\ker\omega$ is said to be
totally umbilical (or, totally geodesic) if Sym$(h) = H g$ ($h=0$, respectively).
Here,
 ${\rm Sym}(h)(X,Y) = \frac12\,(h(X,Y)+h(Y,X))$
is the symmetric second fundamental form of the distribution $\ker\omega$.
The~following Frenet-Serret formulas are true on $U$:
\begin{align}\label{E-Frene}
 \nabla_T\, T=kN,\quad
 \nabla_T\, N=-kT +\tau B,\quad
 \nabla_T\, B=-\tau N .
\end{align}
Using \eqref{E-Frene}, we obtain $k=\<\nabla_T\,T, N\>$ and $\tau = -\<\nabla_T\,B, N\>$.
Recall, see \cite[Eq.~139]{R2021},
\begin{align}\label{E-deta-old}
 d\eta(T,B) = k(\tau - h_{B,N}),\quad
 d\eta(T,N) = T(k) -k h_{N,N}.
\end{align}
With the volume form ${\rm d}\vol_g$ in mind, we derive formulas similar to the following one, see \cite{rw3,rw-gv2}:
 $gv = - \int_M k^2(\tau - h_{B,N})\,{\rm d}\vol_g$,
obtained for foliations by B.L. Reinhart and J.W. Wood in \cite{rw73} with opposite sign by convention.

\begin{prop} We obtain
\begin{align}\label{eq:rwood1}
 & gv^* = -\!\int_M k^2(\tau + h_{N,B})\,{\rm d}\vol_g .
\end{align}
\end{prop}

\begin{proof}
We have $\varphi N=B$ and $\varphi B=-N$. From Frenet-Serret formulas
and $\eta^*=\eta\circ\varphi$,
we find
\begin{align*}
 d\eta^*(T,B) & = T(\eta(\varphi B)) - B(\eta(\varphi T)) - \eta(\varphi[T,B]) = -T(\eta(N)) - k\,\<\varphi[T,B], N\> \\
 & = - T(k) + k\<\nabla_T\,B-\nabla_B\,T, B\> = -T(k) + k\,h_{B,B}, \\
 d\eta^*(T,N) & = T(\eta(\varphi N)) - N(\eta(\varphi T)) - \eta(\varphi[T,N]) = T(\eta(B)) - k\,\<\varphi[T,N], N\> \\
 & = k\,\<B, \nabla_T\,N-\nabla_N\,T\> = k\,(\tau + h_{N,B}).
\end{align*}
 Thus, using \eqref{E-deta-old}, $\eta^*(T)=\eta^*(N)=0$ and $\eta^*(B)=-k$, we calculate
\begin{align*}
 & (\eta^*\wedge d\eta^*)(T,N,B) =
 \eta^*(B)\,d\eta^*(T,N) = - k^2(\tau + h_{N,B}).
\end{align*}
Applying to the above the volume form ${\rm d}\vol_g$ on $(M^3, g)$, we get \eqref{eq:rwood1}.
\end{proof}



\section{The first variation of $gv^*$}

Let a smooth orientable 3-dimensional manifold $M$ be equipped with a non-vanishing vector field~$T$.
Then for any Riemannian metric $g$ on $M$ such that $g(T,T)=1$ there exist a unique 1-form $\omega\in\Lambda^1(M)$
(given by $\omega=g(T,\,\cdot)$)
and a unique $(1,1)$-tensor $\varphi\in End(TM)$ such that
the restriction of $\varphi$ on the plane $\ker\omega$ specifies a right-hand rotation and
$M^3(\varphi,\omega,T, g)$ is an a.c.m. manifold.
Let ${\cal C}(M,T)$ denote the set of all such a.c.m. structures on $(M,T)$.

Let $M^3(\varphi,\omega,T,g)$ be an almost contact manifold, and $g_t\ (|t|<\varepsilon)$ a family of Riemannian metrics on $M^3$ such that $g_0=g$ and $g_t(T,T)\equiv1$. By the above, there exists a unique family $(\varphi_t,\omega_t,T,g_t)$ of a.c.m. structures in ${\cal C}(M,T)$
such that $\varphi_0=\varphi$ and $\omega_0=\omega$.
Denote by dot the $t$-derivative at $t=0$ of any quantity on~$M$.
Since $\oc T=0$ and $\oc g_{T,T}=0$ are true, the~symmetric $(0,2)$-tensor $\overset{\centerdot}g = (d g_t/dt)|_{t=0}$ has five independent components
on a domain $U\subset M$ (where~$k\ne0$):
\[
 \oc g_{T,N},\quad \oc g_{T,B},\quad \oc g_{N,N},\quad \oc g_{N,B},\quad \oc g_{B,B}.
\]
Such variations $g_t$ that generate on $U$ only nonzero components $\oc g_{N,N},\oc g_{N,B}$ and $\oc g_{B,B}$, are called $g^\top$-variati\-ons.
They preserve $T$ and $\ker\omega$; thus produce trivial Euler-Lagrange equations for the functional $gv$,
see \cite{rw3,rw-gv2}.
In contrast, the $g^\top$-variations are essential for the functional $gv^*$ and they will be considered in Section~\ref{sec:4.1}.

Variations $g_t$ of $g$ that generate on $U$ only nonzero components
$\oc g_{T,N}$ and $\oc g_{T,B}$, are called $g^\pitchfork$-variations,
see \cite{rw3,rw-gv2}.
The $g^{\pitchfork}$-variations
will be considered in Section~\ref{sec:4.2}.

Recall the
general variational formula for the volume form, see \cite[p.~162]{R2021}:
\begin{equation}\label{E-dtdvol}
 ({\rm d} \vol_{g})^\centerdot = \frac{1}{2}\,(\tr_g \oc g) \, {\rm d} \vol_g \,.
\end{equation}
Applying
variation $g_t$ to \eqref{E-L-Civita}, one can find how the
connection changes, e.g., \cite[p.~158]{R2021}:
\begin{equation}\label{eq2G}
 2\,\<(\oc\nabla_X\,Y), Z\> = (\nabla^t_X\,{\oc g})(Y,Z) + (\nabla^t_Y\,{\oc g})(X,Z) - (\nabla^t_Z\,{\oc g})(X,Y),
 \quad X,Y,Z\in\mathfrak{X}_M.
\end{equation}

The main goal of this section is the following

\begin{theorem}\label{T-EL-5}
The Euler-Lagrange equations on $U$ of the functional $gv^*$ with respect to all variations in ${\cal C}(M,T)$ are the following:
\begin{align}\label{E-EL-gv0}
\notag
 & \tau + h_{N,B} =0 ,\\
\notag
& T(k) - k\,h_{B,B}=0 , \\
\notag
 & (h_{N,B} + h_{B,N})\,(h_{N,N}+ h_{B,B}) = 0,\\
 & h^2_{N,N} - h^2_{B,B} - h_{N,N}h_{B,B} - T(h_{N,N}) - \frac14\,k^2 = 0.
\end{align}
\end{theorem}

\begin{proof}
Using Euler-Lagrange equations \eqref{E-EL-gv} for  $g^\top$-variations (Section~\ref{sec:4.1}), we simplify Euler-Lagrange equations \eqref{E-EL-gv3} for $g^\pitchfork$-variations (Section~\ref{sec:4.2}), and get~\eqref{E-EL-gv0}.
\end{proof}

\subsection{The $g^\top$-variations}
\label{sec:4.1}


\begin{lemma}
For $g^\top$-variations of metric on $U$, we obtain $\oc\eta=0$ and
\begin{align}\label{E-diff}
\notag
 & \oc N = -\frac12\,\oc g_{N,N}\,N - \oc g_{N,B}\,B, \qquad \oc B = - \frac12\,\oc g_{B,B}\,B , \qquad \oc k = - \frac k2\,\oc g_{N,N} ,
\\
 & (\tau + h_{N,B})^\centerdot = (h_{N,N} - h_{B,B})\oc g_{N,B} +  \frac12\,(\tau + h_{N,B})\oc g_{B,B} - T(\oc g_{N,B}).
\end{align}
\end{lemma}

\begin{proof} Using $\oc T=0$,
we find
\begin{align*}
 & \<\oc N,T\>=\<\oc B,T\>=0,\quad
 \<\oc N,N\>=-\frac12\,\oc g_{N,N},\\
 & \<\oc B,B\>=-\frac12\,\oc g_{B,B},\quad
 \<\oc N,B\>+\<\oc B,N\>=-\oc g_{N,B}.
\end{align*}
Differentiating $\<\nabla_T\,T, N\>=k$, $\<\nabla_T\,T, B\>=0$ and using \eqref{E-Frene} and \eqref{eq2G}, we find
\begin{align*}
 & \oc k = \frac k2\,\oc g_{N,N} + (\nabla_T\,\oc g)_{T,N} -\frac12\,(\nabla_N\,\oc g)_{T,T},\\
 & \< \oc B, N\> = - \oc g_{N,B} - \frac1k\,(\nabla_T\,\oc g)_{T,B} +\frac1{2\,k}\,(\nabla_B\,\oc g)_{T,T}.
\end{align*}
By the above, $\< \oc N, B\> = \frac1{k}\,(\nabla_T\,\oc g)_{T,B} - \frac1{2\,k}\,(\nabla_B\,\oc g)_{T,T}$ is true.
Thus,
\begin{align*}
 & \oc N = -\frac12\,\oc g_{N,N}\,N + \frac1k\,\big((\nabla_T\,\oc g)_{T,B} - \frac12\,(\nabla_B\,\oc g)_{T,T} \big)\,B, \\
 & \oc B = \big(- \oc g_{N,B} - \frac1k\,(\nabla_T\,\oc g)_{T,B} + \frac1{2\,k}\,(\nabla_B\,\oc g)_{T,T} \big)\,N - \frac12\,\oc g_{B,B}\,B .
\end{align*}
Next, using \eqref{E-Frene} and $\oc g_{T,X}=0$, we find
\begin{align*}
\notag
 & (\nabla_N\,\oc g)_{T,T} = (\nabla_B\,\oc g)_{T,T} = 0 , \quad  (\nabla_T\,\oc g)_{T,B} = - k\,\oc g_{N,B},
 \quad (\nabla_T\,\oc g)_{T,N} = - k\,\oc g_{N,N} , \\
 & (\nabla_T\,\oc g)_{N,N} = T(\oc g_{N,N}) - 2\,\tau\,\oc g_{N,B},\quad (\nabla_N\,\oc g)_{N,T} = h_{N,N}\,\oc g_{N,N} + h_{N,B}\,\oc g_{N,B}.
\end{align*}
From the above we find
$\oc N$, $\oc B$ and $\oc k$ in \eqref{E-diff}.
Using \eqref{E-h-tensor} and \eqref{E-Frene}, we calculate
\[
 \tau + h_{N,B} = -\<\nabla_T\,B, N\> + \<\nabla_N\,B, T\> = \<\nabla_T\,N, B\> - \<\nabla_N\,T, B\> = \<[T,N], B\>.
\]
Since the Lie bracket does not depend on metric, we get
\begin{align*}
 & (\tau + h_{N,B})^\centerdot = \<[T,N], B\>^\centerdot = \oc g_{[T,N], B} + \<[T, \oc N], B\> + \<[T,N], \oc B\>  \\
 & = h_{N,N}\,\oc g_{N,B} -\<[T,B], B\>\,\oc g_{N,B} - T(\oc g_{N,B}) +\frac12\,(\tau + h_{N,B})\,\oc g_{B,B} ,
\end{align*}
from which the expression of $(\tau + h_{N,B})^\centerdot$ in \eqref{E-diff} follows.
\end{proof}


Our main result in Section~\ref{sec:4.1} is the following.

\begin{prop}
An a.c.m. manifold $M^3(\varphi,\omega,T,g)$ is critical for
$gv^*$ with respect to $g^\top$-variations
if and only if the following Euler-Lagrange equations hold on $U$:
\begin{align}\label{E-EL-gv}
  & \tau + h_{N,B} =0 ,\qquad
 T(k) - k\,h_{B,B}=0 .
\end{align}
\end{prop}

\begin{proof}
Using \eqref{eq:rwood1} and \eqref{E-dtdvol}, we find
\begin{align*}
 (gv^*)^\centerdot = -\int_M \big\{ (k^2(\tau + h_{N,B}))^\centerdot +\frac{k^2}2\,(\tau + h_{N,B})(\oc g_{N,N} + \oc g_{B,B})\big\} d\vol_g.
\end{align*}
By \eqref{E-diff}, we obtain
\begin{align*}
\notag
 & (k^2(\tau + h_{N,B}))^\centerdot = 2\,k\,\oc k\,(\tau + h_{N,B}) + k^2(\tau + h_{N,B})^\centerdot \\
 & = \frac{k^2}2\,(\tau + h_{N,B})\oc g_{B,B} - k^2(\tau + h_{N,B})\,\oc g_{N,N} + k^2(h_{N,N} - h_{B,B})\oc g_{N,B} - k^2 T(\oc g_{N,B}).
\end{align*}
Using the equalities $T(f) = \Div(f\cdot T) - f\cdot\Div T$
for any function $f\in C^1(M)$ and ${\rm div}\,T = -\tr_g h = -(h_{N,N} + h_{B,B})$, see \cite{R2021}, we get
\begin{align*}
 k^2 T(\oc g_{N,B}) = {\rm div}(k^2\oc g_{N,B}\,T) + \big(k^2(h_{N,N} + h_{B,B}) - 2\,k\,T(k)\big)\,\oc g_{N,B}.
\end{align*}
Therefore,
\begin{align*}
\notag
 & (gv^*)^\centerdot = -\int_M \Big\{ - k^2(\tau + h_{N,B})\,\oc g_{N,N} + k^2(h_{N,N} - h_{B,B})\oc g_{N,B} +  \frac{k^2}2\,(\tau + h_{N,B})\oc g_{B,B} \\
\notag
 & - \big(k^2(h_{N,N} + h_{B,B}) - 2\,k\,T(k)\big)\,\oc g_{N,B} + \frac{k^2}2\,(\tau + h_{N,B})(\oc g_{N,N} + \oc g_{B,B})\Big\} d\vol_g \\
 & = -\int_M \Big\{ \frac{k^2}2\,(\tau + h_{N,B})\,(\oc g_{B,B} - \oc g_{N,N}) + 2\,(k\,T(k) - k^2 h_{B,B})\,\oc g_{N,B} \Big\} d\vol_g .
\end{align*}
Since $\oc g_{N,N},\oc g_{N,B},\oc g_{B,B}$ are arbitrary functions on $M$, this completes the proof.
\end{proof}

\begin{corollary}
Let $\{T,N,B\}$ be an orthonormal frame on a Riemannian manifold $(M^3,g)$ such that the plane field $Span(T,N)$ is tangent to a Riemanian foliation,
$\nabla_T\,T$ is nonzero and parallel to $N$ and each $T$-curve has constant curvature.
Set $\omega=T^\flat$
and define $\varphi$ by $\varphi(T)=0$, $\varphi(N)=B$ and $\varphi(B)=-N$.
Then $(\varphi,\omega,T,g)$ is critical for $gv^*$ with respect to $g^\top$-variations.
\end{corollary}

\begin{proof} Since the plane field $Span(T,N)$ is integrable, the
first equation of \eqref{E-EL-gv} holds.
Since the foliation is Riemannian, we have $\nabla_B\,B=0$. By this,
 the second
equation of \eqref{E-EL-gv}~holds.
\end{proof}

\begin{example}
\rm
We find solutions of \eqref{E-EL-gv} with $k\ne0$ using orthogonal coordinates in $\RR^3$.

(i) Consider
$\RR^3\setminus\{\rho=0\}$ with cylindrical coordinates $(\rho,\phi,z)$.
Set $T=\partial_\phi$, $N=-\partial_\rho$ and $B=\partial_z$. Define $\varphi$ by $\varphi(T)=0$, $\varphi(N)=B$ and $\varphi(B)=-N$.
The $T$-curves are circles in $\RR^3$, the $(T,N)$-surfaces are horizontal planes $\{z=const\}$, and the $B$-curves vertical lines.
Since the distribution Span$(T,N)$ is integrable, the first Euler-Lagrange equation of \eqref{E-EL-gv} is true.
Since $T(k)=0$ and $\nabla_B\,B=0$, also the second Euler-Lagrange equation of \eqref{E-EL-gv} is true.

(ii)  Consider
$\RR^3\setminus\{\rho=0\ {\rm or}\ \theta=\pi/2\}$ with spherical coordinates $(\rho,\theta,\phi)$.
Let $T$-curves be circles that are the intersections of spheres $\{\rho=const\}>0$ with horizontal planes.
Then $k=0$ on the plane $\theta=\pi/2$,
and $k=\infty$ on the axis $\rho=0$.
Set~$B=\partial_\rho$ and $N=B\times T$.
Hence, spheres $\{\rho=const>0\}$ compose a foliation tangent to Span$(T,N)$.
Define $\varphi$ by $\varphi(T)=0$, $\varphi(N)=B$ and $\varphi(B)=-N$.
Therefore, the~Euler-Lagrange equations \eqref{E-EL-gv} are true.
\end{example}

\subsection{The $g^\pitchfork$-variations}
\label{sec:4.2}


\begin{lemma}
For $g^\pitchfork$-variations of metric on $U$, we obtain
\begin{align}\label{E-diff-2}
\notag
 &  \oc k = T(\oc g_{T,N}) - (\tau + h_{N,B})\oc g_{T,B}
 - h_{N,N}\,\oc g_{T,N} ,
 \\
\notag
 & \oc N = -\frac12\,\oc g_{T,N}\,T +\frac1{k}\,\big((\tau - h_{B,N})\oc g_{T,N} -h_{B,B}\,\oc g_{T,B} +T(\oc g_{T,B})
 \big)\,B, \\
\notag
 & \oc B = -\frac12\,\oc g_{T,B}\,T -\frac1{k}\,\big((\tau - h_{B,N})\oc g_{T,N} -h_{B,B}\,\oc g_{T,B} +T(\oc g_{T,B})
 \big)\,N, \\
\notag
 & (\tau + h_{N,B})^\centerdot =
 \frac1k\,(h_{N,N} + h_{B,B})\big( (\tau - h_{B,N})\oc g_{T,N} - h_{B,B}\,\oc g_{T,B} + T(\oc g_{T,B})
 \big) \\
 & + T\big(\frac1k\,\big( (\tau - h_{B,N})\oc g_{T,N} - h_{B,B}\,\oc g_{T,B} + T(\oc g_{T,B})
 \big) \big)
 - \frac k2\,\oc g_{T,B} .
\end{align}
\end{lemma}

\begin{proof}
Since $\{T,N_t,B_t\}$ is an orthonormal frame on $U_t$ for all $g_t$, we get
\begin{align*}
 & \oc g_{T,N} + 2\,\<T, \oc N\>=0,\quad \oc g_{T,B} + 2\,\<T, \oc B\>=0,\\
 & \<\oc N, N\>=0,\quad \<\oc B, B\>=0,\quad \<\oc N, B\> + \<\oc B, N\>=0 .
\end{align*}
Differentiating $\<\nabla_T\,T, N\>=k$, $\<\nabla_T\,T, B\>=0$ and using \eqref{E-Frene} and \eqref{eq2G}, we find
$\oc k$ in \eqref{E-diff-2} and
\begin{align*}
 k\,\< \oc B, N\> =  h_{B,B}\,\oc g_{T,B} - (\tau - h_{B,N})\oc g_{T,N} - T(\oc g_{T,B}) .
\end{align*}
Therefore, $k\,\< \oc N, B\> = -h_{B,B}\,\oc g_{T,B} + (\tau - h_{B,N})\oc g_{T,N} + T(\oc g_{T,B})$
 is true.
From the above we find
$\oc N$, $\oc B$ and $\oc k$ in \eqref{E-diff-2}.
Finally, we get
\begin{align*}
 (\tau + h_{N,B})^\centerdot = \<[T,N], B\>^\centerdot = \oc g_{[T,N], B} + \<[T, \oc N], B\> + \<[T,N], \oc B\>  ,
\end{align*}
from which and known $\oc N$, $\oc B$ the expression of $(\tau + h_{N,B})^\centerdot$ in \eqref{E-diff-2} follows.
\end{proof}

Our main result in Section~\ref{sec:4.2} is the following.

\begin{prop}
An a.c.m. manifold $M^3(\varphi,\omega,T,g)$ is critical for
$gv^*$ with respect to $g^\pitchfork$-variations
if and only if the following Euler-Lagrange equations hold on $U$:
\begin{align}\label{E-EL-gv3}
\notag
 & k\,(2\,\tau + h_{N,B} - h_{B,N})\,( h_{N,N} + h_{B,B}) - k\,(\tau + h_{N,B}) h_{N,N} \\
\notag
 &\qquad - T(k)\,(2\,\tau + h_{N,B} - h_{B,N}) - k\,T(\tau + h_{N,B}) =0 \\
\notag
 &  k\,(h_{N,N} + h_{B,B})^2 - T(k\,(h_{N,N} + h_{B,B})) - k\,(h_{N,N} + h_{B,B}) h_{B,B} \\
 &\qquad + T(T(k)) - T(k)\,h_{N,N}  - k\,(\tau + h_{N,B})^2 - \frac14\,k^3 =0.
\end{align}
\end{prop}

\begin{proof}
For a $g^\pitchfork$-variation, using \eqref{eq:rwood1} and \eqref{E-dtdvol} with $\tr_g\,\oc g=0$, we find
\begin{align*}
 (gv^*)^\centerdot
 =  -\int_M (k^2(\tau + h_{N,B}))^\centerdot\,d\vol_g.
\end{align*}
By \eqref{E-diff}, we obtain
\begin{align}\label{Eq-kth-2}
\notag
 & (k^2(\tau + h_{N,B}))^\centerdot = 2\,k\,\oc k\,(\tau + h_{N,B}) + k^2(\tau + h_{N,B})^\centerdot \\
\notag
 & = 2\,k\,(\tau + h_{N,B}) \big(T(\oc g_{T,N}) - (\tau + h_{N,B})\oc g_{T,B}
 - h_{N,N}\,\oc g_{T,N}\big) \\
\notag
 & + k\,(h_{N,N} + h_{B,B})\big( (\tau - h_{B,N})\oc g_{T,N} - h_{B,B}\,\oc g_{T,B} + T(\oc g_{T,B}) \big) \\
 & + k^2 T\big(\frac1k\,\big( (\tau - h_{B,N})\oc g_{T,N} - h_{B,B}\,\oc g_{T,B} + T(\oc g_{T,B}) \big) \big) - \frac12\,k^3\oc g_{T,B} .
\end{align}
Using the equalities
$T(\alpha) = \Div(\alpha\cdot T) - \alpha\cdot\Div T$ for any function $\alpha\in C^1(M)$ and ${\rm div}\,T = -\tr_g h = -(h_{N,N} + h_{B,B})$, see \cite{R2021}, we get for any functions $\alpha,\beta,\gamma\in C^1(M)$:
\begin{align*}
 & \alpha T(\oc\beta T) = \Div(\alpha\beta T) + \big[\alpha(h_{N,N}+h_{B,B}) - T(\alpha)\big]\oc\beta,\\
 & \alpha T(\beta T(\oc\gamma)) = \Div([\alpha\beta\,T(\oc\gamma) + \alpha\beta(h_{N,N} + h_{B,B})\,\oc\gamma]\,T ) \\
 & + \big[\alpha\beta(h_{N,N}+h_{B,B})^2 + T(\beta T(\alpha)) - \beta T(\alpha)(h_{N,N}+h_{B,B}) - T(\alpha\beta(h_{N,N}+h_{B,B})) \big]\oc\gamma .
\end{align*}
Therefore, with some functions $f_i$ on $M$ we get
\begin{align*}
 & 2\,k\,(\tau {+} h_{N,B})\,T(\oc g_{T,N}) = \big[2\,k\,(\tau {+} h_{N,B})\,( h_{N,N} {+} h_{B,B}) {-} 2\,T(k\,(\tau {+} h_{N,B}))\big]\,\oc g_{T,N}
 {+} \Div f_1, \\
 & k\,(h_{N,N} + h_{B,B})\,T(\oc g_{T,B}) = \big[k\,(h_{N,N} + h_{B,B})^2 - T(k\,(h_{N,N} + h_{B,B}))\big]\,\oc g_{T,B} + \Div f_2, \\
 & k^2 T\big(\frac1k\,\big( (\tau - h_{B,N})\oc g_{T,N} - h_{B,B}\,\oc g_{T,B}\big)\big) = k\,(\tau - h_{B,N})(h_{N,N} + h_{B,B})\,\oc g_{T,N} \\
 & - k\,h_{B,B}(h_{N,N} + h_{B,B}),\oc g_{T,B}) + \Div f_3 ,\\
 & k^2 T(\frac1k\,T(\oc g_{T,B})) = \big[k\,(h_{N,N} + h_{B,B})^2 - T(k\,(h_{N,N} + h_{B,B})) - 2\,T(k)\,(h_{N,N} + h_{B,B}) \\
 & + 2\,T(T(k))\big]\,\oc g_{T,B} + \Div f_4.
\end{align*}
Integrating \eqref{Eq-kth-2} and using the above and the Divergence Theorem, gives
\begin{align*}
 & (gv^*)^\centerdot = -\int_M \Big\{ 2\,k\,(\tau + h_{N,B})\,T(\oc g_{T,N}) + k\,(h_{N,N} + h_{B,B})\,T(\oc g_{T,B}) \\
 & + k^2 T\big(\frac1k\,\big( (\tau - h_{B,N})\oc g_{T,N} - h_{B,B}\,\oc g_{T,B} \big) \big) + k^2 T\big(\frac1k\,T(\oc g_{T,B} ) \big) \\
 & - 2\,k\,(\tau + h_{N,B})(\tau + h_{N,B})\oc g_{T,B} - 2\,k\,(\tau + h_{N,B}) h_{N,N}\,\oc g_{T,N} \\
 & + k\,(h_{N,N} + h_{B,B})\big( (\tau - h_{B,N})\oc g_{T,N} - h_{B,B}\,\oc g_{T,B}  \big)
  - \frac12\,k^3\oc g_{T,B}
 \Big\} d\vol_g \\
 & = -4\!\int_M \!\!\Big\{ \Big[ k(2\tau {+} h_{N,B} {-} h_{B,N})( h_{N,N} {+} h_{B,B}) {-} k(\tau {+} h_{N,B}) h_{N,N}
 {-} T(k(\tau {+} h_{N,B}))\Big]\oc g_{T,N} \\
 & + \Big[ k\,(h_{N,N} + h_{B,B})^2 - T(k\,(h_{N,N} + h_{B,B})) - k\,h_{B,B}(h_{N,N} + h_{B,B}) \\
 & - T(k)\,(h_{N,N} + h_{B,B}) + T(T(k)) - k\,(\tau + h_{N,B})^2 - \frac14\,k^3 \Big]\oc g_{T,B} \Big\} d\vol_g .
\end{align*}
Since $\oc g_{T,N},\oc g_{T,B}$ are arbitrary functions on $M$, this completes the proof of \eqref{E-EL-gv3}.
\end{proof}

\section{Critical almost contact metric manifolds}

Using the Euler-Lagrange equations of Theorem~\ref{T-EL-5}, we get the following

\begin{prop}\label{C-geod1}
Any contact metric manifold $M^3(\varphi,\omega,T,g)$ is critical for the action $gv^*$ with respect to all variations in ${\cal C}(M,T)$.
\end{prop}

\begin{proof}
Since $T$ is a geodesic vector field ($k=0$), see \cite{Blairsurvey}, then $U=\emptyset$ and \eqref{E-EL-gv0} become trivial.
\end{proof}

We list the defining conditions
(formulated in terms of the covariant derivatives $\nabla\varphi$, $\nabla\,T$ and $\nabla\omega$)
of any a.c.m. manifold $(M,\varphi,T,\omega, g)$ which falls in $C_5\oplus C_{12}$ or in its subclasses, see \cite{CF-2019}:
\begin{align*}
  & C_5\oplus C_{12}: \
  (\nabla_X\,\varphi)Y {=} \beta(\<\varphi X,Y\>T {-} \omega(Y)\varphi X) {-} \omega(X)((\nabla_T\,\omega)(\varphi Y)T {+}\omega(Y)\varphi(\nabla_TT) ), \\
  & C_5: \ \qquad\quad (\nabla_X\,\varphi)Y = \beta\,(\<\varphi X,Y\>T - \omega(Y)\,\varphi X),\quad \beta=const>0, \\
  & C_{12}: \quad\qquad (\nabla_X\,\varphi)Y=-\omega(X)((\nabla_T\,\omega)(\varphi Y)\,T +\omega(Y)\varphi(\nabla_TT)) , \\
 & |C|=C_5\cap C_{12}: \quad \nabla\,\varphi = 0 \quad {\rm (cosymplectic\ manifolds)}.
\end{align*}
The vanishing of the tensor $h$, see \eqref{E-h-tensor}, that is $\ker\omega$ defines a totally geodesic foliation, means that the considered manifold belongs to $C_{12}$.
The vanishing of $\nabla_T\,T$ means that the considered manifold belongs to $C_5$, namely, it is a $\beta$-Kenmotsu manifold.
Thus, we get the following

\begin{prop}\label{C-geod2}
Any 3-dimensional a.c.m. manifold of a class $C_{5}$ is critical for the action $gv^*$ with respect to all variations in ${\cal C}(M,T)$.
The set of 3-dimensional a.c.m. manifolds of a class $C_{12}$ that are critical for the action $gv^*$ with respect to all variations in ${\cal C}(M,T)$
coincides with~$|C|$.
\end{prop}

Any a.c.m. structure with a geodesic vector field $T$ (e.g., Propositions~\ref{C-geod1} and \ref{C-geod2})
can be called ``trivial" solution to
\eqref{E-EL-gv0}.
What are non-trivial solutions to \eqref{E-EL-gv0}?

\begin{lemma}
If an a.c.m. manifold $M^3(\varphi,\omega,T,g)$ is critical for the action $gv^*$ with respect to all variations in ${\cal C}(M,T)$,
then the distribution Span$(T,N)$ on $U$ is integrable.
Moreover, if the distribution $\ker\omega$ is either totally umbilical or integrable,
then also the distribution
Span$(T,B)$ on $U$ is integrable.
\end{lemma}

\begin{proof}
Using the equalities \eqref{E-h-tensor} and $\tau = \<\nabla_T\,N, B\>$,
we rewrite the Euler-Lagrange equation \eqref{E-EL-gv}$_1$
as $\<[T,N], B\>=0$; thus, the distribution Span$(T,N)$ is integrable.
By~the conditions and equalities $h_{N,B}+h_{B,N}=0$ (for totally umbilical $\ker\omega$) or $h_{N,B}-h_{B,N}=0$ (for integrable $\ker\omega$)
and $\<[B,T],N\>=\tau-h_{B,N}$, the second claim is true.
\end{proof}

\begin{prop}\label{Cor-01}
Let an a.c.m. manifold $M^3(\varphi,\omega,T,g)$ with integrable totally umbilical distribution $\ker\omega$ and the mean curvature $H\not\equiv0$
be critical for the action $gv^*$ with respect to all variations in ${\cal C}(M,T)$.
Then the distributions Span$(T,N)$ and Span$(T,B)$ on $U$ are also integrable and the Euler-Lagrange equations \eqref{E-EL-gv0}
are reduced to the following: $\tau = 0$ and
\begin{align}\label{E-EL-gv-umb}
 T(k) = k\,H , \quad T(H) = - H^2 - (1/4)\,k^2 .
\end{align}
\end{prop}

\begin{proof}
By conditions, $h_{N,N}=h_{B,B}=H$ and $h_{N,B}=h_{B,N}=0$. Thus, the Euler-Lagrange equations \eqref{E-EL-gv0} are reduced to $\tau=0$ and \eqref{E-EL-gv-umb}.
\end{proof}

The Euler-Lagrange equations \eqref{E-EL-gv-umb} on $U$ can be seen on any $T$-curve (paramete\-rized by $s$) as the dynamical system of two ODEs for
functions
$k(s)$ and $H(s)$:
\begin{align}\label{E-sys-kH}
 (d/ds) k = k\,H , \quad (d/ds) H = - H^2 - (1/4)\,k^2.
\end{align}

\begin{lemma}\label{L-003}
The general solution of the system \eqref{E-sys-kH} has the following form:
\begin{align*}
  s = \int(H^4+C_1)^{-1/2}\,dH + C_2,\quad
  k(s) =\pm 2\,\sqrt {- \frac {\rm d}{{\rm d}s}\,H(s) - H^2(s)} ,
\end{align*}
or, in Taylor series form with $k(0)=k_0\ne0$ and $h(0)=H_0$,
\begin{align}\label{E-series-kH}
\notag
 & H (x) = H_0 - \Big( H_0^{2} +{\frac{ k_0^{2}}{4}} \Big) x + H_0^{3}{x}^{2}
 - H_0^{2}\Big( H_0^{2} + {\frac{k_0^{2}}{4}}\Big) {x}^{3}
 + O({x}^{4}), \\
 & k(x) = k_0 +k_0 H_0\,x - \frac { k_0^{3}}{8}\,{x}^{2}-{\frac {k_0^{3} H_0}{8}}\,{x}^{3}
 + O({x}^{4}) .
\end{align}
\end{lemma}

\begin{proof}
From \eqref{E-sys-kH} we get the following ODEs:
\begin{align}\label{E-sys-kH-2}
 H'' = 2\,H^3,\quad k'' = - k^3/4.
\end{align}
Let $k(0)=k_0\ne0$, $H(0)=H_0$ be the values of solutions of \eqref{E-sys-kH} at $t=0$. Then the values of the first derivatives
of solutions of \eqref{E-sys-kH-2} at $t=0$ should be
 $k'(0)=k_0 H_0$
 and
 $H'(0)=- H_0^2 - \frac14\,k_0^2$,
see Fig.~\ref{Fig-01} with $k_0=1$, $H_0=0$ and $|s|<3.4$, and the series expansion \eqref{E-series-kH} is valid.
\begin{figure}[h]
\centering
\includegraphics[width=0.95\textwidth]{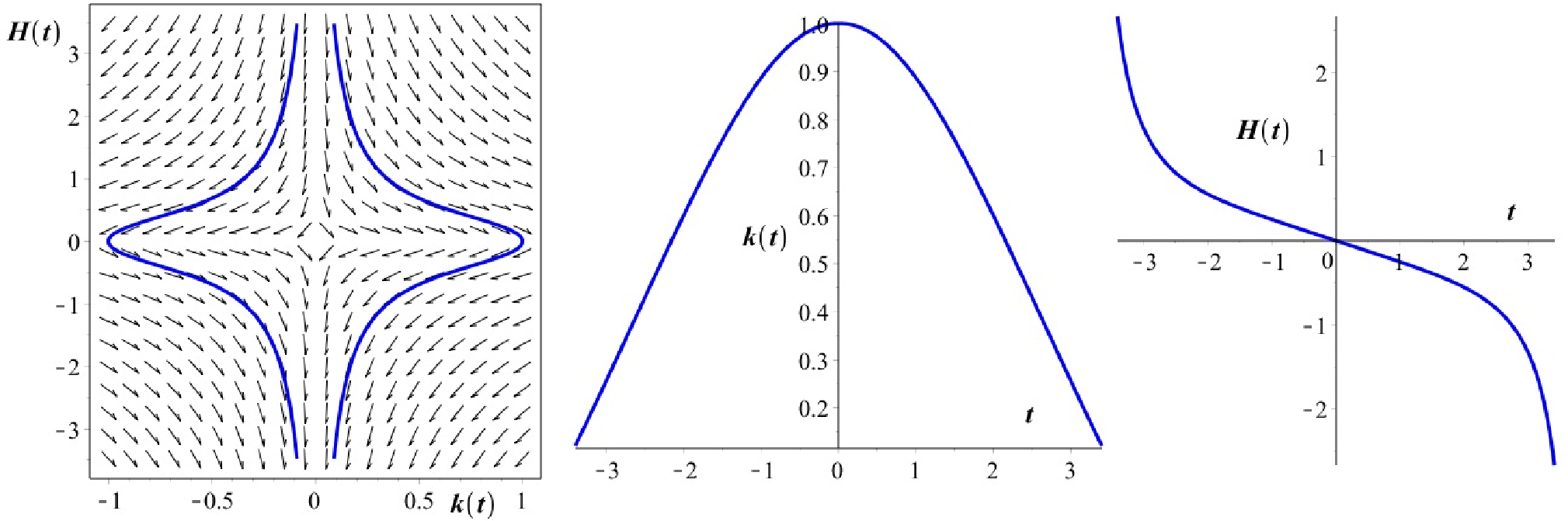}
\caption{Solution $k(s),H(s)$ of \eqref{E-sys-kH} with $k_0=1$ and $H_0=0$ for $|s|<3.4$}
\label{Fig-01}
\end{figure}
\end{proof}

\begin{theorem}\label{T-EL-02}
There exist 3-dimensional a.c.m. manifolds of a class $C_{5}\oplus C_{12}$ critical for the action $gv^*$ with respect to all variations in ${\cal C}(M,T)$. These manifolds have integrable distributions Span$(T,N)$ and Span$(T,B)$ on $U$, $\tau=0$ and
are presented locally as double twisted products $B\times_{(u,v)} I$,
where the functions ${H} = -2\,\<\nabla(\log u), T\>$ and $k = -\<\nabla(\log v), N\>\ne0$
satisfy~\eqref{E-EL-gv-umb}.
\end{theorem}

\begin{proof}
By \cite{CF-2019,f-2013}, it is sufficient to build a critical double twisted structure $B\times_{(u,v)}I$
with metric  $g=(u^2 g_B)\oplus(v^2 ds^2)$ on a domain $M=B\times I$ in $\RR^3(x,y,s)$,
where $B=(-1,1)\times(-1,1)\subset\RR^2(x,y)$ and $u,v: M\to \RR$ are positive smooth functions.
The~second fundamental form $h$ and the mean curvature $H$ of the leaves $B\times\{s\}$ and the curvature $k$ of the fibers $\{x\}\times I$
are given~by, see~\cite{pr,R2021}.
\begin{align}\label{Eq-Hk}
 {h} = -\<\nabla(\log u), T\>\,g^\top,\quad
 {H}\cdot T = -2\,\nabla^\bot(\log u),\quad
 k\cdot N = -\nabla^\top(\log v),
\end{align}
where $^\top$ is the projection on $\ker\omega$ and $^\bot$ is the projection on Span$(N)$.
The~{leaves} $B\times\{s\}$ are totally umbilical in $(M,g)$.
For a critical a.c.m. structure on $(M,T)$, by Proposition~\ref{Cor-01},
we get $\tau=0$ and the distributions Span$(T,N)$ and Span$(T,B)$ on $U$ are integrable.
Using Lemma~\ref{L-003}, we restore functions $k$ and $H$ on $B\times I\subset\RR^3(x,y,s)$
from their arbitrary initial values $k_0\ne0$ and $H_0$ on $B$ (for $s=0$).
Next, we will show existence of appropriate functions $u,v$ on $B\times I$.
Let us assume that $u=u(s)$ depends on one variable, thus, by \eqref{Eq-Hk}, $T$ is parallel to coordinate vector $\partial_s$
and we can restore $u$ from its arbitrary initial values at $s=0$ by integration along $s$-coordinate lines.
Then, we assume that $v=v(x)$ depends on one variable, thus, by \eqref{Eq-Hk}, $N$ is parallel to coordinate vector $\partial_x$
and we can restore $v$ from its arbitrary initial values at $x=0$ by integration along $x$-coordinate lines.
\end{proof}

\section{Conclusion}

In the article, we applied the calculus of variations approach to finding best a.c.m. structures for a given manifold.
We defined a new Godbillon-Vey type functional $gv^*$ for a 3-dimensional a.c.m. manifold,
found its Euler-Lagrange equations for all variations preserving the Reeb vector field and
constructed critical 3-dimensional a.c.m. manifolds having a double-twisted product structure, i.e., solutions
belonging to the class $C_{5}\oplus C_{12}$ according to Chinea-Gonzalez classification.

In further work, we intend to study the critical a.c.m. manifolds (for $gv^*$) with nonintegrable distribution~$\ker\omega$.
The~following tasks also seem interesting:
study counterparts $gv^*_1=\int_M \eta\wedge d\eta^*$ and $gv^*_2=\int_M\eta^*\wedge d\eta$ of $gv^*$;
calculate the second variations of $gv^*$ and $gv^*_i$ and find their~extrema.

We also intend to study the multidimensional case of $gv^*$.
For any a.c.m. manifold of dimension $2n+1\ge5$, one may define one-form $\eta^* = \eta\circ\varphi=-k\,(\varphi N)^\flat$,
and analogously to the functionals $gv_s = \int_M \eta\wedge(d\eta)^p\wedge(d\omega)^{n-p}$ for all $p\ge1$ in \cite{rw3,rw-gv2},
consider the following functionals:
\begin{equation}\label{E-gv-n}
 gv^*_s = \int_M \eta^*\wedge (d\eta^*)^p\wedge(d\omega)^{n-p}, \quad 1\le p\le n .
\end{equation}
A~question arises: \textit{what a.c.m. manifolds,
e.g., in $\bigoplus_{1\le i\le 5} C_{i}\oplus C_{12}$ due to Chinea-Gonzalez classification,
are optimal for functionals \eqref{E-gv-n} with respect to all variations in ${\cal C}(M,T)$}?




\begin{thebibliography}{}

\bibitem{as2015}
T. Asuke, Transverse projective structures of foliations and infinitesimal
derivatives of the Godbillon-Vey class. Int. J. of Math. 26(4), 2015, 29 pp.

\bibitem{Blairsurvey}
D.E. Blair, {A survey of Riemannian contact geometry}, Complex manifolds, {6} (2019), 31--64.

\bibitem{bbs23}
B. Bayour, G. Beldjilali, M.\,L. Sinacer, Almost contact metric manifolds with certain condition. Ann. Glob. Anal. Geom. 64\,(2) (2023), 12\,pp.

\bibitem{CF-2019}
S. de\,Candia, M. Falcitelli, Curvature of $C_5\oplus C_{12}$-manifolds. Mediterr. J. Math., 16, 105 (2019).

\bibitem{CGD-90}
D. Chinea, C. Gonz\'{a}lez,
 A classification of almost contact metric manifolds, Ann. Mat. Pura Appl. 156(4) (1990), 15--36.

\bibitem{f-2013}
M. Falcitelli, A class of almost contact metric manifolds and double twisted products. Math. Sci. Appl. E-Notes (MSAEN), 1, 36--57 (2013)


\bibitem{gv71}
C. Godbillon, J. Vey, {Un invariant des feuilletages de codimension 1},
C. R. Acad. Sci. Paris S{\'e}r A-B, 273 (1971), A92--A93.


\bibitem{kenmotsu1972}
K. Kenmotsu, A class of almost contact Riemannian manifolds, T\^{o}hoku Math. J., 24 (1972), 93--103.

\bibitem{pr}
R.\,Ponge, and H.\,Reckziegel: {Twisted products in pseudo-Riemannian geometry}, Geom. Dedi\-cata,  {48} (1993), 15--25

\bibitem{rw73}
B.L. Reinhart and J.W. Wood, {A metric formula for the Godbillon--Vey invariant for foliations}, Proc. Amer. Math. Soc., 38, No. 2 (1973), 427--430.

\bibitem{rw3}
V. Rovenski and P. Walczak, {Variations of the Godbillon-Vey invariant of foliated 3-manifolds},
Complex Analysis and Operator Theory, 13(6), (2019), 2917--2937.

\bibitem{rw-gv2}
V. Rovenski and P. Walczak, {A Godbillon-Vey type invariant for a 3-dimensional manifold with a plane field}.
Differential Geom. and its Applications, 66, (2019), 212--230.

\bibitem{R2021}
V. Rovenski and P. Walczak, {Extrinsic geometry of foliations}, Birkh\"{a}user, Progress in Mathe\-matics, 339, 2021, 319 pp.

\end{thebibliography}
\end{document}